\documentclass{interact} 
\usepackage{latexsym}
\usepackage{mathtools} 
\usepackage{anyfontsize} 
\usepackage{tikz} 
\usetikzlibrary{cd} 
\usepackage{color}
\usepackage{hyperref}
\usepackage{url}
\usepackage{breakurl}
\newcommand{\bburl}[1]{\textcolor{blue}{\url{#1}}}

\makeatletter
\newcommand{\monthyear}[1]{%
  \def\@monthyear{\uppercase{#1}}}
\newcommand{\volnumber}[1]{%
  \def\@volnumber{\uppercase{#1}}}
\makeatother

\theoremstyle{plain}
\numberwithin{equation}{section} 
\newtheorem{thm}{Theorem}[section] 
\newtheorem{theorem}[thm]{Theorem}
\newtheorem{lemma}[thm]{Lemma}
\newtheorem{example}[thm]{Example}
\newtheorem{definition}[thm]{Definition}

\newtheorem{remark}[thm]{Remark}
\newtheorem{comment}[thm]{Comment}

\numberwithin{table}{section} 
\numberwithin{figure}{section}

\begin{document}

\monthyear{Month Year}
\volnumber{Volume, Number}
\setcounter{page}{1}

\title{Proof of Convergence of a Laplace Expansion Algorithm For Calculating Recursions Satisfied by a Family of Determinants} 

\author{
\name{Russell Jay Hendel \thanks{Email address: rhendel@towson.edu}}
\affil{Department of Mathematics,
                Towson University, 
                Towson, MD USA }}

\maketitle

{\bf Article type}: research 
\bigskip

\begin{abstract}
In Evans and Hendel's recent proof of  an outstanding conjecture on the resistance distances of a family of linear 3-trees, a key technique in the proof was calculating the recursion satisfied by a family of determinants. The underlying algorithm employed to prove the conjecture converged (i.e., terminated) in the particular case studied, and the paper presented an open question on when such a procedure converges in general. This paper proves  the convergence of a Laplace expansion procedure for an arbitrary family of determinants of banded, square, Toeplitz matrices. A comparison of the procedure presented in this paper, the paper by Evans and Hendel, and a paper by Jia, Yang, and Li is presented. 
\end{abstract}

\begin{keywords}
recursion; families of matrices; determinants;    Toeplitz
\end{keywords}

\section{Introduction}\label{sec:introduction}
A recent paper \cite{Proof}  proving an open conjecture of Barrett, Evans, and Francis \cite{Barrett}  regarding the asymptotic behavior of the resistance distance of a straight linear 3-tree, introduced a formal procedure for calculating the characteristic polynomial, or equivalently, the recursion satisfied by a family of determinants. The procedure converged (terminated) in that particular case and the authors, in the conclusion of the paper, state, ``The formal procedure introduced seems to have independent interest in its own right and may be applicable to a wider variety of graph families whose adjacency matrices are banded (or nearly banded). Whether the procedure converges, as well as how one might improve its efficiency, remain open questions.'' 

The main theorem (Theorem \ref{the:main}) of this paper provides a specific sequence of Laplace expansions that, when applied to an arbitrary    family of determinants whose underlying matrix family is  banded, square, and Toeplitz,  will always  converge, allowing the calculation of the recursion satisfied by the family of determinants.
The next five sections  provide the  notation, important background material, definitions, basic lemmas, and an illustrative example, after which the main theorem of this paper is precisely formulated and proven. 

\section{Notation and Conventions}\label{sec:notation}
This section and the next present the definitions, notations, and conventions used throughout the paper. We begin with the following definition.

\begin{definition}
Integer $R$ is said to be the  \textit{Toeplitz order}     of a  banded, square, Toeplitz matrix $M,$   if $R$ is the smallest integer such that for $i > R$,  $M_{i,1} =0$ and $M_{1,i}=0.$ 
\end{definition}

\begin{example} It follows that tri-diagonal and penta-diagonal Toeplitz matrices have Toeplitz orders 2 and 3 respectively. \end{example}

If $A$ and $B$ are sets of indices (or singleton indices) and $M$ is an arbitrary $n \times n$ matrix, $n \ge 1,$ then we let $(A;B) M$ indicate the matrix obtained from $M$ by deleting the rows whose indices are in $A$ and deleting the columns whose indices are in $ B$.     The operator $(0;0)$ indicates the identity operator (mnemonically, we remove row 0 and column 0 that do not exist; hence, we leave the family as is). The equality of operators is defined in the usual way: $(A;B)=(A';B')$ means $(A;B)M = (A';B') M.$ Operators act from right to left (the inner operator first). 

\textbf{Conventions.} Throughout the paper, 
\begin{equation}\label{equ:AB}
\left. \begin{aligned}
	\qquad & \text{ A general operator is notationally indicated as } (A;B) & \qquad\\
	\qquad & \text{ with } A=(A_1,\dotsc, A_n), \text{ and } B=(B_1,\dotsc, B_n);  & \qquad \\
	\qquad & \text{ without loss of generality we assume }  & \qquad \\ 
	\qquad & A_1 \le A_2 \le \dotsc \le A_n \text{ and } B_1 \le B_2 \le \dotsc \le B_n.  & \qquad\\
	\qquad & \text{ We omit parentheses and  write } (A_1, \dotsc, A_n; B_1, \dotsc, B_n).  & \qquad
\end{aligned}
\right \}
\end{equation} 

 If $M$ is a matrix, then, as usual, we let $M_{i,j}$ refer to the entry of $M$ in row $i$ and column $j,$ and  we let both $\det(M)$ and $M^*$ denote the determinant of $M$.

Reference \cite{Bapat} gives more traditional notation. However, this paper selects a notation that is more suitable  for   sequences of Laplace expansions. The following computations for a banded, square, Toeplitz matrix of Toeplitz order at least 4 illustrate some subtleties related to the use of a sequence of operators.
\begin{equation}\label{equ:1through4}
\left.  \begin{aligned}
(i) \qquad & (1;1)(1;1)M =&(1,2;1,2)M \\
(ii) \qquad & (1;1)(1;2)M =&(1,2;1,2)M\\
(iii) \qquad & (1;3)(1;2)(1;2)
M=& (1,2,3; 2,3,5)M \\
(iv) \qquad & (2;3)(1;1)(1;1)M =& (1,2,4;1,2,5)M=\\
(v) \qquad & (1;5)(1;5)(1;5)M= & (1,2,3;5,6, 7)M.
\end{aligned}
\right\}
\end{equation} 

We verbally derive \eqref{equ:1through4}(iii), the proof of the others being similar and hence omitted.  If you successively, three times, remove  the first row of an $n  \times n$ matrix, $M,$  with  $n \ge 6,$ and then the first remaining row of the resulting matrix, you have equivalently removed the first 3 rows of the original matrix, $M$. Furthermore, after removing the second column of $M$, the second column of the resulting  matrix is the third column of $M$. Similarly,   after removing the second and third columns of $M$, the third remaining column of the resulting matrix is the fifth column of $M$. The resulting matrix, whether after applying the three operations on the left-hand side of (iii) to $M$ or applying the single operation on the right-hand side of (iii), is of size $n-3 \times n-3$.

Throughout the paper,  \textit{expansion} refers to a  Laplace expansion.
The  expansion of a matrix $M$ along the first row is given by  
\begin{equation}\label{equ:basicstepfamily}
    \det(M) = \displaystyle \sum_{\text{all $j$}} (-1)^{j+1} M_{1;j}  \det((1,j) M). 
\end{equation}
There is  an analogous   formula for expansion along the first column.

\section{Matrix Families}\label{sec:matrixfamilies}
Equation \eqref{equ:basicstepfamily} defines   the expansion for an individual matrix. We would like to restate \eqref{equ:basicstepfamily} in terms of a family of matrices. However, a problem arises because in \eqref{equ:basicstepfamily} the size of the underlying matrices $(1,j)M$ on the right-hand side is one less than the size of matrix $M$ on the left-hand side. To avoid degeneracy problems, it is preferred that the sizes of all matrices in an identity on a family of matrices be identical. 

One way of fixing this is to use the backward shift operator, $y,$ which operates on a sequence of numbers $\{s_n\}_{n \ge c}$ for some constant $c,$ by $y(s_{n}) =s_{n+1}.$ If matrix $m^{\{n\}}$ of size $n$ 
belongs to a family of matrices, say
$M=(m^{\{k\}}))_{k \ge 1}$, with the superscripts indicating size, then  for an arbitrary summand  on the right-hand side of  \eqref{equ:basicstepfamily} applied to the member of the family $m^{\{n\}}$ we have 
$$
(-1)^{j+1}
m^{\{n\}}_{1,j}
 det((1,j)m^{\{n\}})
=
(-1)^{j+1}
m^{\{n\}}_{1,j}
y (det((1,j)m^{\{n+1\}})).
$$
By  using the $y$ operator, the underlying matrix of the right-hand side of this last equation, $(1,j)m^{\{n+1\}},$  has size $n,$ the same size as the underlying matrix that is being expanded.  (We could have alternatively used the backward shift operator on matrices instead of numbers.)

Finally, to avoid the clutter of excess parentheses, we notationally let
$$
    y \; (det((1,j) m)) = y (1,j) m^*, 
$$
where the superscript asterisk indicating determinants is applied after the operator $(1,j)$ is applied. Since applying $(1,j)$ to a determinant has no meaning, there is no ambiguity in this.

 Using these conventions, we may now rewrite \eqref{equ:basicstepfamily} for an entire  family of matrices, $M,$  either as 
\begin{equation*}
 M = \displaystyle 
\sum_{  
\text{all $j$}
}  
 (-1)^{j+1}   M_{1,j}  y\det((1,j) M), 
\end{equation*}
or, using the superscript asterisk notation,  
\begin{equation}\label{equ:familynotation2}M^* = \displaystyle 
\sum_{\text{all $j$}  
}  
 (-1)^{j+1}   M_{1,j}  y(1,j) M^*, 
\end{equation}
which is a purely formal notation indicating that for each size $n$ matrix of the underlying family,  the resulting equation in individual determinants (if well-defined) is true.

$n,$ the index of the matrix or determinant family, typically ranges over the positive integers; however, for a variety of technical reasons, it may range over positive integers greater than some constant. The range will always be clear from the context and need not be known for the development of the theory.

\begin{example} The following should fully clarify the various subtleties in the notation. Let 
$M =\{m^{\{k\}}\}_{k \ge 1}$ with superscripts indicating the matrix size. We can reformulate  \eqref{equ:1through4} in terms of the family $M$ as follows:
\begin{equation}\label{equ:1through4F}
\left.  \begin{aligned}
(i) \qquad & (1;1)(1;1)M =&(1,2;1,2)M =& (0,0)M=M\\
(ii) \qquad & (1;1)(1;2)M =&(1,2;1,2)M =&(0,0)M=M\\
(iii) \qquad & (1;3)(1;2)(1;2)
M=& (1,2,3; 2,3,5)M \\
(iv) \qquad & (2;3)(1;1)(1;1)M =& (1,2,4;1,2,5)M=& (2;3)M\\
(v) \qquad & (1;5)(1;5)(1;5)M= & (1,2,3;5,6, 7)M.
\end{aligned}
\right\}
\end{equation} 

Continuing our discussion on \eqref{equ:1through4}(iii) as it applies to \eqref{equ:1through4F}(iii), we note that for each $k \ge 1,$ the following three matrices are equal and of size $ k$:
\begin{itemize}
\item $(2,3)(1,1)(1,1) m^{\{k+3\}}$
\item $(1,2,4;1,2,5) m^{\{k+3\}}$
\item $(2,3) m^{\{k+1\}}$
\end{itemize}
 We can further verbally illustrate this by letting $k=3.$ Then the matrices resulting either from deleting rows 1,2,4 and columns 1,2,5 from $m^{\{6\}}$ or from 
deleting row 2 and column 3 from $m^{\{4\}}$ are identical of size 3.
\end{example}

 The point of the family matrix notation introduced is that by simply stating 
 $$
 (1,2,4;1,2,5)M = (2,3)M
 $$
 we avoid the problem of having to continuously address the details of the sizes.

 \section{Basic Illustrative Example and Lemmas }\label{sec:illustrativeexample}
Because of the many new notations used, it is expositionally easier if we interweave an illustrative example with the lemmas needed to prove the main theorem.

Towards that end, for this and the next section, Sections \ref{sec:illustrativeexample}-\ref{sec:review},  let $M$ be the general, banded, square, Toeplitz family of matrices of Toeplitz order $R=3$. To avoid problems of degeneracy we assume that the family begins indexing with the member of size $2R=6$. 
\begin{equation}\label{equ:memberofsize6}
\text{The  member   $m$ of the family $M$
of size  6}=
\begin{psmallmatrix}a & b & c & 0 & 0 & 0\\ d & a & b & c & 0 & 0\\ e & d & a & b & c & 0 \\0 &e & d & a & b & c\\
0 & 0 & e & d & a &b \\
0 & 0 & 0 & e & d & a.
\end{psmallmatrix}.
\end{equation}

The recursion satisfied by the family of determinants of this matrix family was first discovered in \cite{Pentadiagonal} which inspired this paper. However, \cite{Pentadiagonal} did not present a systematic method for deriving the recursion satisfied by the corresponding family of determinants,  which is generalizable to other families of banded, square, Toeplitz matrices.

This paper presents a systematic procedure for obtaining the recursion; the algorithm introduced consists of a sequence of expansions, with the first expansion applied to each member of $M.$ Each expansion
introduces new matrix families and equations in the determinants of matrix families.  We store the equations in a queue, 
\textbf{QEquations}, and we store the new matrix families, which must be further expanded, in a queue, \textbf{QTodo.}  
  The process is organized by initializing \textbf{QTodo} with $M$ and setting \textbf{QEquations} to empty. 

\textbf{Expansion 1.} We expand all members of  $M$ across the first row using \eqref{equ:familynotation2}. The resulting equation in determinants
of the matrix families is stored in \textbf{QEquations}.
\begin{equation}\label{equ:eq1}
	M^*  = ay (1;1)M^*  - b y (1;2)M^* + cy(1;3)M^*= 
    ay M^* - by (1;2)M^* + cy (1;3)M^*.
\end{equation}

Equation \eqref{equ:eq1} uses the following identity on the family of determinants:$(1,1)M^* = M^*$. Equation \eqref{equ:1through4F}(i) is a similar equation. 
Verbally, this identity  observes that removing the first row and column of a size $n$ Toeplitz matrix results in a size $n-1$ Toeplitz matrix from the same family. Equation \ref{equ:1through4F})(iv) shows another similar identity: Removing rows 1,2, and 4 and columns 1,2, and 5 of a size $n$ Toeplitz matrix is equivalent to removing row 2 and column 3 of the size $n-2$ Toeplitz matrix of the same family.     Using the notation conventions of \eqref{equ:AB}, we can formulate the underlying commonalities among these identities. 

\begin{lemma}[Reduction] (a)
If for some $s, 1 \le s \le R-1,
$ 
\begin{equation}\label{equ:defofs}
A_i=i=B_i =i, 1 \le i \le s,
\end{equation}
then 
 $(A;B) = (A',B'),$ \text{ with }  \begin{equation}\label{equ:howreduction}
 C'_{i-s}=C_i - s, \text{ for } C \in \{A,B\} \text{ and } s+1 \le i \le R.
 \end{equation} \\
 (b) If $A_i=i=B_i, 1 \le i \le R,$ then $(A,B)=(0,0).$
 \end{lemma}

 \begin{example} (a) Applying the lemma with $s=2$   to the operator $(1,2,4;1,2,5)$ proves \eqref{equ:1through4F}(iv). \\
 (b) Letting $R=2$ proves
\eqref{equ:1through4F}(i).\end{example}
 
If the hypotheses of the Reduction lemma hold,   we will say that $(A;B)$ is \textit{reducible} to $(A',B')$ or, equivalently, that $(A',B')$ is a \textit{reduction} of $(A;B)$. 
Similarly, the statement that $(A;B)$ is \textit{not reducible} or \textit{irreducible} means that the hypotheses of the Reduction lemma do not hold.

Equation \eqref{equ:eq1} introduces two new matrix families, $(1;2)M, (1;3)M$ and   is therefore placed in \textbf{QTodo}. Matrix families are stored in \textbf{QTodo} as operators, which operate on $M. $                
 
To complete the bookkeeping for Expansion 1, we let 
\begin{equation}\label{equ:eq1E1}
	E(1) = \{(1;2), (1;3) \}
\end{equation} 
indicate the irreducible operators introduced in Expansion 1, and let
\begin{equation}\label{equ:eq1A1}
	E_{3,1}= \{2,3 \}
\end{equation} 
indicate the column components   of all operators in $E(1);$ the 3 and 1 in the subscript of $E$ refer to the Toeplitz order and index of the expansion, respectively.  These sets are used in the proof of the main theorem. 

\textbf{Expansion 2.} The algorithm now requires us to expand each matrix family in \textbf{QTodo} along their first row. To accomplish this, we apply 
each of the operators $(1,k), k \in \{1,2,3\}$ to each of the operators in $E(1) $.

First, we expand $(1;2)M$, obtaining the following equation in determinant families:
\begin{multline}\label{equ:eq2a}
	(1;2)M^* = d y(1;1)(1;2)M^*- by(1;2)(1;2)M^*+cy(1;3)(1;2)M^* \\= dy M^* - by (1,2;2,3) M^* +cy(1,2;2,4)M^*.
\end{multline}
The simplification $dy(1;1)(1;2)M^*=dy M^*$ is in fact
\eqref{equ:1through4F}(ii). The other simplifications in \eqref{equ:eq2a} can be derived using the same reasoning used to derive \eqref{equ:1through4F}(i)-(v).   	

Similarly, we expand
\begin{multline}\label{equ:eq2b}
	(1,3)M^* = dy (1;1)(1;3)M^* - ay (1;2)(1;3)M^* + cy (1;3)(1;3) M^*\\= dy (1,2)M^*-ay(1,2;2,3)M^*+ cy (1,2;3,4)M^*,
\end{multline}
which results from the simplifications $(1;1)(1;3)M^*=(1,2;1,3)M^*=(1;2)M^*,$  $(1;2)(1;3)M^*=(1,2;2,3)M^*,$ and $(1;3)(1;3)M^*=(1,2;3,4)M^*$ which follow either from the Reduction lemma or can be derived using the same reasoning used to derive \eqref{equ:1through4F}(i)-(v).

The irreducible operators introduced in Expansion 2 are
\begin{equation}\label{equ:eq2E2}
	E(2) = \{(1,2;2,3), (1,2;2,4),(1,2;3,4) \}.
\end{equation} 
The corresponding column components are 
\begin{equation}\label{equ:eq2A2}
	E_{3,2}=	  \{(2,3), (2,4),(3,4) \}.
\end{equation} 

To complete the bookkeeping for this step: (i) the matrix families based on the operators in $E(1)$ are removed from \textbf{QTodo}, (ii)  the operators in $E(2)$ are placed in \textbf{QTodo}, and (iii) equations
 \eqref{equ:eq2a} and \eqref{equ:eq2b} are added to \textbf{QEquations}. 

\textbf{Expansion 3.} Prior to performing Expansion 3, we  note that if we successively perform $R-1$ operations on any member of the family $M,$ with each operation removing the top row, then the first column of the resulting matrix has $e$ in its upper-left corner with the remainder of the column 0. Formally, we can state the following lemma.

\begin{lemma}\label{lem:rminus1} Let $M'$ be an arbitrary family of matrices of Toeplitz order $R \ge 2,$ and let $m'$ be a member of it. Applying to $m'$ successively $R-1$ times operators of the form $(1,b),$ with $1 \le b \le R,$ results in a matrix   whose first column   is 0 except possibly for the upper left corner which has $m'_{R,1}.$  \end{lemma} 

For later on, in the proof of the main theorem, we also include the following lemma. whose notation and formulation follow the conventions established in \eqref{equ:AB}. 

\begin{lemma}\label{lem:bge2} Let $M'$ be an arbitrary family of matrices of Toeplitz order $R \ge 2,$ and let $m'$ be any member of it. Further, let $(A;B)$ indicate the equivalent operator resulting from applying, to $m',$  $s$ times, $1 \le s \le R-1,$  operators of the form $(1,b),1 \le b \le R$. Then \\
(a)  $A=(1,2,\dotsc,s);$   \\
(b) If additionally, $(A;B)$ is irreducible, then $B_1 \ge 2;$\\
(c) If (a) and (b) hold and  $k$ satisfies $1 \le k \le R,$ then $(1,k)(A;B)$ is reducible if and only if $k=1$. 
\end{lemma}
\begin{proof} (a) Part (a)  is simply the observation that if you successively remove the first row of a matrix $s$ times then equivalently you have removed the first $s$ rows.

(b) To prove the assertion about $B_1,$ we assume to the contrary that $B_1=1$ and derive a contradiction. By Part (a), $A_1=1.$ But if $A_1=1=B_1$
then by the Reduction Lemma, $(A;B)$ is reducible, contradicting the hypothesis that it is irreducible. This contradiction proves that our assumption $B_1=1$ is incorrect; in other words, $B_1 \ge 2$.

(c) Part (c) follows from our definition of reducible and the fact that $A_1=1, B_1 \neq 1$.
\end{proof}

\begin{example} Lemma \ref{lem:bge2}(a) is   illustrated by \eqref{equ:1through4F} (i)-(v).
\end{example} 
 
At Expansion 3, \textbf{QTodo} has the three operators listed in \eqref{equ:eq2A2}.  Therefore, it follows from
 Lemma \ref{lem:rminus1} that at Expansion 3 we may expand each member of \textbf{QTodo} along the first column
which involves one non-zero entry, which by\eqref{equ:memberofsize6} equals $e.$ Doing so, we obtain the following identities in the family of determinants of $ M$:
\begin{equation}\label{equ:eq3}
\left.  \begin{aligned}
	(1,2;2,3)M^*&=ey(1;1)(1,2;2,3)M^* &=  ey(1,2,3;1,2,3)M^* &=   eyM^*			\nonumber \\ 
    (1,2;2,4)M^* &=
	ey(1;1)(1,2;2,4)M^* &=  ey(1,2,3;1,2,4)M^* &=  ey(1;2)M^*		\\
    (1,2;3,4)M^* &=  
	ey(1,1)(1,2;3,4)M^* &=  ey(1,2,3;1,3,4)M^* &=   ey(1,2;2,3)M^*.	\nonumber
\end{aligned}
\right\}
\end{equation}

\begin{comment} The reduced operators - $(1;2)$ and $(1,2;2,3)$ - appearing on the right-hand side of the identities in \eqref{equ:eq3} have been encountered in previous expansion steps (\eqref{equ:eq1A1},\eqref{equ:eq2A2} and  therefore, at Expansion Step 3, further identities are not added to \textbf{QTodo}. But then
\textbf{QTodo} is empty and consequently the process has
successfully converged (terminated). \end{comment}

 Expansion Step 3 is completed by adding the three identities in \eqref{equ:eq3} to \textbf{QEquations}.

 \section{Review of Solving Simultaneous Equations of Determinant Families}\label{sec:review}
\textbf{QEquations} now has 6 equations. The methods and underlying theorem for solving this system are presented in   \cite{Proof}. However, for the sake of completeness, we briefly review the underlying principles of the solution process. Two techniques, \textit{elimination by substitution} and \textit{multiplication by operators} are used to solve:  

\textbf{Elimination by Substitution.} We may start the solution process with \eqref{equ:eq1}, the equation resulting from the first expansion; we may then eliminate $(1;2)M$ and $(1;3)M$ from \eqref{equ:eq1}   using
\eqref{equ:eq2a} - \eqref{equ:eq2b}; we may further  eliminate the remaining operators    using \eqref{equ:eq3}. These eliminations by substitution yield the following end result.
\begin{equation}\label{equ:multiplysubstitute}
M^* = ay M^* -bd y^2 M +b^2 e y^3 M^*-ace y^3 M^* + c^2 e^2 y^4 M^* - bce y^3 (1;2)M^* + cdy^2 (1;2)M^*. 
\end{equation}

 \textbf{Operator Multiplication.} As can be seen, the substitution technique by itself  cannot eliminate all the operators, as $(1;2)$ remains. However, starting with \eqref{equ:eq2a} and simplifying using \eqref{equ:eq3} we obtain,
\begin{equation*} 
(1;2)M^* = dy M^* - bey^2 M^* + ce y^2(1;2)M^*.
\end{equation*}
We can collect like terms to obtain the equivalent identity
\begin{equation}\label{equ:multiplysubstitute2}
(1-ce y^2) (1;2) M^* = dy M^* - be y^2 M^*.
\end{equation}

Equation \eqref{equ:multiplysubstitute2} motivates
 operating on the left and right-hand sides of \eqref{equ:multiplysubstitute} by the operator $(1-ce y^2),$ because this operation will then allow    the elimination of $(1;2)M$ using \eqref{equ:multiplysubstitute2}. 
 
Upon completing this and gathering like terms,  we   re-derive the result of \cite{Pentadiagonal}:   the family of determinants, $M,$ satisfies the order 6 recursion:
$$
G_n - a G_{n-1} + (bd - ce) G_{n-2} + (2ace - b^2e - cd^2) G_{n-3} + 
 ce (bd - ce) G_{n-4}- ac^2e^2 G_{n-5}
 + c^3e^3 G_{n-6}=0.
$$

\section{ Pseudocode and the Main Theorem}\label{sec:pseudocode}

The pseudocode for this Expand procedure is as follows:

\begin{verbatim}
EXPAND PROCEDURE  (The procedure takes a matrix family M of Toeplitz order, R > 1,
and produces, as needed, a set of new matrix families and 
a set of equations in the determinants of the matrix families. 
New matrix families are indicated by matrix operators applied to M. 
Throughout this procedure, the instruction to reduce refers to Lemma 4.1.)
INITIALIZE:
    QTodo = {M} 
    QEquations = { }
    Expansionindex=1

WHILE Not IsEmpty(QTodo)
    FOR EACH q in QTodo
        DeleteFrom(QTodo, q)
        If ExpansionIndex<=R-1 then 
          Expand q along first row
          For k =1 to R 
            If (1,k)q reducible then
                Reduce it
            Else place (1,k)q in QTodo
          Next
          ExpansionIndex++
        Else
            Expand q along 1st column
            Apply any possible reductions
        End if  
        Place resulting expanded equation in QEquations
    END FOR EACH
END WHILE
\end{verbatim}

  We can now unambiguously formulate the main theorem by referring to the pseudocode, which requires a particular sequence of expansions (to wit, expansion along the first row until Expansion $R-1$ when we expand along the first column). 

\begin{theorem}[Main]\label{the:main} Let $M$ be an arbitrary, banded, square, family of Toeplitz matrices of Toeplitz order $R \ge 2.$ The Expand procedure terminates after $R$ expansions. Using the techniques of elimination by substitution and operator multiplication, the solution to the resulting set of equations in families of determinants can then be solved   to yield the recursion satisfied by this family.
\end{theorem}

\section{A Preliminary Lemma }

The main theorem is proven using Lemma
\ref{lem:fundamental} stated below. Prior to stating it, recall our notation:  $E_{R,k}$ contains all column components $B,$ of operators $(A;B)$ used in Expansion step $ k$. For $R=3,$ the sets  
 $E(k)$ and $E_{R,k}$  
 are illustrated for $k=1 $ in \eqref{equ:eq1E1} - \eqref{equ:eq1A1} and for 
$k=2$ in \eqref{equ:eq2E2}–\eqref{equ:eq2A2}. The sets $E_{3,k}, 1 \le k \le 2,$ show a distinct pattern, which is formulated in the following lemma.    

\begin{lemma}\label{lem:fundamental}For any fixed $R \ge 2,$ and for $k, 1 
\le k \le R-1,$  
\begin{equation}\label{equ:toprove}
    E_{R,k} = \{(b_1,\dotsc,b_k) :
        2 \le b_1 <b_2<\dotsc < b_k \le
        R+k-1\}.
\end{equation}
\end{lemma}

 We divide the proof of \eqref{equ:toprove} into two cases:
\begin{equation}\label{equ:toprovein}
    E_{R,k} \subset \{(b_1,\dotsc,b_k) :
        2 \le b_1 <b_2<\dotsc < b_k \le
        R+k-1\},
\end{equation}
and
\begin{equation}\label{equ:toproveout}
    E_{R,k} \supset \{(b_1,\dotsc,b_k) :
        2 \le b_1 <b_2<\dotsc < b_k \le
        R+k-1\}.
\end{equation}

Both proofs are by induction. We begin with the proof of \eqref{equ:toprovein}.  

\begin{proof}
 \textit{Base Case.} For  arbitrary fixed $R \ge 2,$  the pseudocode requires that the first expansion be along the first row by applying the operators $(1,m), 1 \le m \le R$. By the Reduction lemma, $(1;1)=(0;0)$ is reducible while $(1;m),m \ge 2$ is not reducible. Since at the first expansion $(1;2), (1;3), \dotsc, (1;R)$ are not reducible, we have 
\begin{equation}\label{equ:toprovebase}
    E_{R,1} = \{2,\dotsc,R\}.
\end{equation}

\begin{example}  Equation \eqref{equ:eq1A1} illustrates this for $R=3.$ 
\end{example}

For later use, we note that \eqref{equ:toprovebase} is an exact equality and therefore can also serve as a base case in the proof of \eqref{equ:toproveout}.

\textit{Induction assumption.} For an induction assumption, we assume \eqref{equ:toprovein} to be true for fixed $R \ge 2$ and for any $k,$ with $1 \le k \le m \le R-2,$ for some $m \ge 1 $. The base case shows that for $m=1$ for arbitrary $R \ge 3$, the induction assumption is satisfied. (For $R=2$ it is satisfied vacuously).

\textit{Induction step.} With the notation as in the induction assumption, we must show that \eqref{equ:toprovein} holds for the case $m+1$.

However, since $m \le R-2,$ the pseudocode requires, at Expansion step $m+1,$ expansion along the first row, which is accomplished  by applying the operators $(1,p), 1 \le p \le R,$ to the current matrix, say $(A;B)M$. By Lemma \ref{lem:bge2}, since $(A;B)$ is in \textbf{QTodo}, and hence, by the definition of \textbf{QTodo} is irreducible, we have
$$A=(1,2,\dotsc,m),$$ 
and
$$
    2 \ge B_1.
$$
Moreover, by the induction assumption we have
$$
 2 \le B_1 \le \dotsc \le B_m \le R+m-1.
$$

Select some $p,$ $1 \le   p \le R,$  apply the operator $(1,p)$  to the matrix family $(A;B)M$ in \textbf{QTodo}, and let
$(A';B')=(1,p)(A;B).$ 
By Lemma \ref{lem:bge2}(c), $(A',B')$ is irreducible if $p \ge 2 $.  Then again by
Lemma \ref{lem:bge2}, we have
$$
A'=(1,2,\dotsc,m+1), \qquad
B'_1 \ge 2.
$$

Therefore, to complete the induction step we must show
$$
    B'_{k+1} \le R+(m+1)-1 = R+m.
$$
However, $B$ has length $m;$ in other words,   $B$   eliminates $m$ columns from $ M$. Since $p \le R,$ it follows that the $p$-th remaining column in $(A';B')M = (1,p)(A;B)M$ must lie in the first $R+m$ columns of $M$ as was to be shown.   
\end{proof}

To complete the proof of \eqref{equ:toprove}, we must also prove 
\eqref{equ:toproveout}. This can be done constructively. Prior to stating the proof, we provide an illustrative example.

\begin{example}\label{exa:constructive}

For this example, let
$M$ be an initial matrix family, 
let $R=10,$ and suppose the current expansion step is  $m=3$.  We further suppose that 
\textbf{QTodo} has the operator $(A;B)=(1,2,3;b_1=3,b_2=5,b_3=7).$ The following identity  
\begin{equation}\label{equ:constructive}
(1,6)(A;B) = (1,2,3,4;3,5,7,9)
\end{equation}
is easily justified verbally as follows:
If we have already removed three columns from $M$ then the 6-th column of the resulting matrix is the 9-th column of the original matrix.  Generalizing, if  at expansion step $m, 1 \le m \le R,$ 
 \textbf{QTodo} has the matrix $(A;B) M,$ and using the notation of \eqref{equ:AB}, we are given a $t, B_m < t \le R,$ then if we wish to find an operator $(1,x)$ such that $(1,x)(A,B) = 
 (1,2, \dotsc,k+1;B_2,\dotsc,B_m,B_{m+1}=t),$ then it suffices to let $x=t-m$.  
 \end{example}

 Using this fact, we proceed to  prove \eqref{equ:toproveout}.

 \begin{proof} The base case is given by
 \eqref{equ:toprovebase}; for an induction assumption we assume \eqref{equ:toproveout} true for fixed $R \ge 2$ and for any $k,$ with $1 \le k \le m \le R-2,$ for some $m \ge 1 $. For the induction step, we must show that \eqref{equ:toproveout} is true for the case $m+1$; but this has already been shown above in Example \ref{exa:constructive}.

This completes the proof of 
\eqref{equ:toproveout}, which, when coupled with the proof of \eqref{equ:toprovein}, completes the proof \eqref{equ:toprove}.
\end{proof}

\begin{example} Equation \eqref{equ:1through4F}(v) illustrates  that applying the operator $(1,R)$ $k$ times to a matrix (with $1 \le k \le R) $ is equivalent to the operator 
$(1, \dotsc,k; 
R, R+1, \dotsc, R+k-1),$ showing that the upper bound \eqref{equ:toprove} is the best possible.
\end{example}

\section{ Proof of the Main Theorem}\label{sec:proofofmaintheorem} By the pseudocode, at Expansion Step $R-1,$ the items in \textbf{QTodo} are irreducible. 
Hence, by Lemma \ref{lem:bge2},
$A=(1,2,\dotsc,R-1),$ and $B_1 \ge 2.$

At Expansion step $R,$ the pseudocode requires expansion along the first column, which, by  Lemma \ref{lem:rminus1}, has one non-zero entry. Thus, the expansion at Step $R$ has one summand. Therefore, we consider $(1;1)(A;B)=(A';B').$ 
We claim
\begin{equation}\label{equ:AprimeatR}
A'=(1,\dotsc,R)
\end{equation}
and
\begin{equation}\label{equ:BprimeatR}
B'=(1, B_1,\dotsc,B_{R-1}).
\end{equation}
Equation \eqref{equ:AprimeatR} follows from Lemma \ref{lem:bge2}. To see \eqref{equ:BprimeatR}, note that $B_1 \ge 2$ and   by our conventions stated 
in \eqref{equ:AB},   $B_1 \le B_2 \le \dotsc \le B_{R-1}.$ Consequently, column 1 was not deleted by $(A;B),$ and therefore  $(1;1)(A;B)$ will simply delete both column 1 and also the columns listed in $B.$

By \eqref{equ:AprimeatR} and \eqref{equ:BprimeatR}, $A'_1=1=B'_1 $. Therefore, the hypotheses of the Reduction Lemma hold and $(A';B')$ can be reduced to an equivalent operator with say $(A',B')=(A'',B'').$ To complete the proof of the main theorem we need only show that $(A'';B'')$ was encountered in a previous expansion step and therefore will not enter \textbf{QTodo}. For if we show that, then \textbf{QTodo} will be empty, because there are no new entries. In other words, the process terminates, proving convergence.

To prove that $(A'';B'')$ was previously encountered we  
use \eqref{equ:defofs}-\eqref{equ:howreduction} in the Reduction lemma, which describes how reductions take place. There are two cases 
to consider.

\textit{Case 1.} Suppose $A'_i = B'_i$ for $1 \le i \le R.$ Then by Part (b) of  the Reduction Lemma   $(A';B')$ may be reduced to $(A''';B'')=(0;0)$, which has been previously encountered.

\textit{Case 2.} There is a largest $s < R$ such that $A'_i=B'_i, 1 \le i \le s,$ but $A'_{s+1} \neq B'_{s+1} $.  
By Lemma \ref{lem:bge2},   $A'_i=i, 1 \le i \le R,$ and 
hence $A'_{s+1} = s+1. $
Consequently, by \eqref{equ:howreduction} of the Reduction Lemma,  $A''_1 = A'_{s+1}-s = s+1-s=1.$
By the hypothesis of this Case 2, $B'_{s+1} \neq A'_{s+1}$ and therefore, since $A'_{s+1}=s+1,$ we must have $B'_{s+1} > s+1 $.
But then by \eqref{equ:howreduction}, $B''_1 = B'_{s+1} - s  \ge s+2-2=2$.

A similar argument shows that $B''_{R-s} < R+(R-s)-1.$  Hence, by \eqref{equ:AB}, $2 \le B''_k \le 2R-s-1, \text{ for } 1 \le k \le R-s.$
By Lemma \ref{lem:fundamental}, it then follows that $(A'';B'')$ is a member of $E(R-s ),$ proving that it was previously encountered.
This completes the proof of the main theorem.

\begin{example} Equation \eqref{equ:eq3} showing how how all items in \textbf{QTodo} are reduced at Expansion step $R=3$ to previously encountered operators, nicely illustrates the arguments just presented for completion of the proof of the main theorem.
\end{example}

\begin{remark} The main theorem also mentions that the resulting set of equations in determinants can be solved. The method of solving   the equations in \textbf{QEquations} is presented in \cite{Proof} and was reviewed in Section
\ref{sec:review}.  
\end{remark}

\section{Comparison of Approaches}\label{sec:comparison}

The Toeplitz order, $R=3,$ illustrative example is solved in this paper in Section \ref{sec:illustrativeexample}, was treated in \cite{Pentadiagonal}, and can also be solved using the software presented in the arXiv version of \cite{Proof}. 

The major points of contrast and  comparison between these three methods are as follows: 
\begin{enumerate}
\item [(i)] The methods presented in this paper are applicable to the general, Toeplitz, square, banded family of matrices while    \cite{Pentadiagonal} simply performed some ad-hoc matrix operations to obtain the recursion, without those operations being generalizable to other determinant families. 
\item [(ii)] The identification of previously encountered matrix families in this paper is done through the Reduction lemma, while   the algorithm presented in \cite{Proof} requires manually checking matrices at each step to verify prior encounters. 
\item [(iii)] The main theorem of this paper guarantees convergence after $R$ expansions; contrastively, the convergence in \cite{Proof} was simply a fortuitous accident for that particular example.
\item[(iv)] Both the methods of this paper and \cite{Proof} introduced half a dozen new matrix families, and both terminate after $R=3$ Laplace Expansions. While the idea of \cite{Proof} to expand along both rows and columns and to check for transposes seems to point to greater efficiency, this does not seem to matter in practice. 
\end{enumerate}

\section*{Disclosure statement}
The author declares no conflicts of interest.


\medskip
\noindent MSC2020: 11B39, 11C20, 65F40

\end{document}